\documentclass[a4paper]{article}

\usepackage{
amsmath,
amsthm,
amscd,
amssymb,
}
\usepackage{comment}
\usepackage{tikz}
\usetikzlibrary{positioning,arrows,calc}
\usepackage{xspace}

\usepackage{graphicx}
\usepackage{subcaption}
\usepackage{hyperref}
\usepackage{mathtools}
\usepackage{wasysym}
\hypersetup{
	colorlinks = true, %
	urlcolor = blue, %
	linkcolor = red, %
	citecolor = blue %
}
\usepackage{url}
\usepackage{ifthen}

\makeatletter
\newtheorem*{rep@theorem}{\rep@title}
\newcommand{\newreptheorem}[2]{%
\newenvironment{rep#1}[1]{%
 \def\rep@title{#2~\ref{##1}}%
 \begin{rep@theorem}}%
 {\end{rep@theorem}}}
\makeatother

\theoremstyle{plain}
\newtheorem{lemma}{Lemma}[section]
\newtheorem*{lemma*}{Lemma}
\newtheorem{claim}{Claim}[section]
\newtheorem{definition}{Definition}
\newtheorem{corollary}{Corollary}[section]
\newtheorem{proposition}{Proposition}
\newreptheorem{proposition}{Proposition}
\newtheorem{theorem}{Theorem}[section]
\newreptheorem{theorem}{Theorem}
\newtheorem*{theorem*}{Theorem}
\newtheorem*{proposition*}{Proposition}

\newtheorem*{conjecture*}{Conjecture}
\newtheorem{question}{Question}
\newtheorem{remark}{Remark}

\newtheoremstyle{derp}
{3pt}
{3pt}
{}
{}
{\upshape}
{:}
{.5em}
{}
\theoremstyle{derp}
\newtheorem{example}{Example}

\newcommand{\symb}[1]{\mathtt{#1}}

\newcommand{\Z}{\mathbb{Z}}

\newcommand{\N}{\mathbb{N}}

\newcommand{\cP}{\mathcal{P}}

\newcommand{\lk}{\mathrm{link}}

\newcommand{\bb}{\mathbf{b}}
\newcommand{\br}{\gre{\mathbf{r}}}
\newcommand{\bc}{\blu{\mathbf{c}}}

\newcommand{\bP}{\mathbf{P}}

\newcommand\xqed[1]{%
  \leavevmode\unskip\penalty9999 \hbox{}\nobreak\hfill
  \quad\hbox{#1}}
\newcommand\qee{\xqed{$\fullmoon$}}

\newcommand{\col}{\mathrm{col}}

\newcommand{\tail}{\mathrm{tail}}

\newcommand{\gre}[1]{\textcolor{orange}{#1}}
\newcommand{\blu}[1]{\textcolor{blue}{#1}}
\newcommand{\bB}{\mathbf{B}}
\newcommand{\bS}{\mathbf{S}}
\newcommand{\bD}{\mathbf{D}}
\newcommand{\bL}{\mathbf{L}}
\newcommand{\bI}{\mathbf{I}}
\newcommand{\bT}{\mathbf{T}}

\newcommand{\bseed}{\mathbf{seed}}

\newcommand{\cE}{\mathcal{E}}
\newcommand{\cV}{\mathcal{V}}

\newcommand\vartextvisiblespace[1][.3em]{%
  \mbox{\kern.1em\vrule height.3ex}%
  \vbox{\hrule width#1}%
  \hbox{\vrule height.3ex}
}

\newcommand{\cycle}[2]{
\begin{tikzpicture}[scale=#2]
    \def\r{1.5}    

    \foreach \i in {1,...,#1} {
        \node[circle, fill=black, inner sep=1.5pt] 
              (v\i) at ({90-(\i-1)*360/#1}:\r) {};
    }

    \foreach \i [evaluate=\i as \j using {int(mod(\i,#1)+1)}] in {1,...,#1} {
        \draw (v\i) -- (v\j);
    }
\end{tikzpicture}
}

\newcommand{\cocycle}[2]{
\begin{tikzpicture}[scale=#2]
    \def\r{1.5}    

    \foreach \i in {1,...,#1} {
        \node[circle, fill, inner sep=1.5pt] 
              (v\i) at ({90-(\i-1)*360/#1}:\r) {};
    }

    \foreach \i in {1,...,#1} {
		\pgfmathtruncatemacro{\startj}{\i+2}
		\pgfmathtruncatemacro{\sodij}{#1+1}
		\ifthenelse{\startj < \sodij}{
			\foreach \j in {\startj,...,#1} {
			\ifnum\i=1
            \ifnum\j=#1
            \else
                \draw (v\i) -- (v\j);
            \fi
        \else
            \draw (v\i) -- (v\j);
        \fi
			}}{}
	}
\end{tikzpicture}
}

\newcommand{\copath}[2]{
\begin{tikzpicture}[scale=#2]
    \def\r{1.5}    

    \foreach \i in {1,...,#1} {
        \node[circle, fill, inner sep=1.5pt] 
              (v\i) at ({90-(\i-1)*360/#1}:\r) {};
    }

    \foreach \i in {1,...,#1} {
		\pgfmathtruncatemacro{\startj}{\i+2}
		\pgfmathtruncatemacro{\sodij}{#1+1}
		\ifthenelse{\startj < \sodij}{
			\foreach \j in {\startj,...,#1} {
			
            \draw (v\i) -- (v\j);
			}}{}
	}
\end{tikzpicture}
}

\newcommand{\copathnumbers}[2]{
\begin{tikzpicture}[scale=#2]
    \def\r{1.5}    

    \foreach \i in {1,...,#1} {
        \node[circle, inner sep=1.5pt] 
              (v\i) at ({90-(\i-1)*360/#1}:\r) {\i};
    }

    \foreach \i in {1,...,#1} {
		\pgfmathtruncatemacro{\startj}{\i+2}
		\pgfmathtruncatemacro{\sodij}{#1+1}
		\ifthenelse{\startj < \sodij}{
			\foreach \j in {\startj,...,#1} {
			
            \draw (v\i) -- (v\j);
			}}{}
	}
\end{tikzpicture}
}

\newcommand{\petersen}{
\begin{tikzpicture}[scale=0.5]
    \foreach \i in {1,...,5} {
        \node[circle, fill, inner sep=1.5pt] 
              (v\i) at ({90-(\i-1)*360/5}:1.5) {};
		\pgfmathtruncatemacro{\kek}{\i+5}
\node[circle, fill, inner sep=1.5pt] 
              (u\i) at ({90-(\i-1)*360/5}:0.75) {};
    }

    \draw (v1) -- (v2) -- (v3) -- (v4) -- (v5) -- (v1);
        \draw (u1) -- (u3) -- (u5) -- (u2) -- (u4) -- (u1);
        \draw(u1)--(v1);
                \draw(u2)--(v2);
                        \draw(u3)--(v3);
                                \draw(u4)--(v4);
                                        \draw(u5)--(v5);
\end{tikzpicture}
}

\title{Self-simulability of graph products}

\author{
Kan\'eda Blot \and Ville Salo
}

\begin{document}
\maketitle

\begin{abstract}
A group is self-simulable if all its computable actions admit SFT covers, which means roughly that they can be implemented with finitely many tiling constraints. We prove that a graph product of infinite finitely-generated groups is self-simulable if and only if its defining graph has no disconnecting clique consisting of amenable groups. In particular, a right-angled Artin group (a.k.a.\ a graph group) is self-simulable if and only if the defining graph has no disconnecting clique. As an application, we obtain that a graph product of infinite finitely-generated groups splits (algebraically, or in a certain geometric sense) over an amenable subgroup if and only if the graph has a disconnecting clique consisting of amenable groups.
\end{abstract}

\section{Introduction}

Let $\Gamma$ be a group, and $A$ a finite set, called the \emph{alphabet}. The main object of study in the field of symbolic dynamics is the \emph{subshift}, namely a topologically closed and $\Gamma$-invariant subsystem of the translation action $\Gamma \curvearrowright A^\Gamma$. Our convention for the translation action is that a group $\Gamma$ acts on the left on a subshift $\mathcal S \subset A^\Gamma$ by $\Gamma \times \mathcal S \ni (g,c) \mapsto (h \mapsto c(g^{-1}h))$.

An interesting class of subshifts are \emph{subshifts of finite type} or \emph{SFTs}, obtained by removing from $A^\Gamma$ those points whose $\Gamma$-orbit intersects a clopen set (this can be seen as ``forbidding finitely many patterns'').

A factor map (surjective shift-equivariant continuous map) from an SFT to a $\Gamma$-system is known as an \emph{SFT cover}, and we call systems having an SFT cover \emph{SFT covered}. Finding SFT covers is a standard method of studying hyperbolic toral automorphisms \cite{Bo73,LiMa95} as well as the boundaries of hyperbolic groups \cite{CoPa06}. SFT covered subshifts are called \emph{sofic shifts} \cite{We73a}, and they are in themselves an interesting class of dynamical systems.

An SFT cover can be seen as a form of finite presentation of the system. Namely, an SFT is fully described by a clopen set which can be described by a finite amount of data if $\Gamma$ is finitely-generated. In many cases also the factor map admits a finite description. In particular, this happens in the case of sofic shifts, where the Curtis-Hedlund-Lyndon theorem \cite{CeCo10} provides a combinatorial characterization of factor maps. More generally finite descriptions of the factor map exist for all expansive factors \cite{Fr87}, and for many non-expansive ones \cite{KoSa21}.

SFT covered systems are typically a large class of dynamical systems. Hochman showed in \cite{Ho09} a strong theorem, which we state in simplified form:

\begin{theorem*}[Hochman]
\label{th:Hochman}
Let $\Z^d \curvearrowright X$ be any effective action, where $X$ is an effectively closed subset of Cantor space. Then the $\Z^{d+2}$-system on $X$ where the two new generators act trivially, is SFT covered.
\end{theorem*}

Here an \emph{effectively closed} subset of Cantor space $\{0,1\}^\omega$ refers to a subset obtained by removing a countable set of cylinders enumerated by a Turing machine (that is, a $\Pi_1^0$ set), and an \emph{effective action} is one where each group element acts by a computable homeomorphism.

By an example of Jeandel (see \cite{BaSaSa26}), the statement of Hochman's theorem fails to hold if $\Z^{d+2}$ is replaced by $\Z^{d+1}$.\footnote{If we restrict $\Z^d \curvearrowright X$ to itself be a subshift, then it does hold \cite{DuRoSh10,AuSa13}.}

Hochman's theorem suggests the following definition:

\begin{definition}
Suppose $\phi : \Gamma \to \Delta$ is a surjective group homomorphism. To every $\Delta$-system, we may associate its \emph{pullback (along $\phi$)}, namely the $\Gamma$-system by defining $g \cdot x = \phi(g) \cdot x$. We say that $\Gamma$ \emph{simulates $\Delta$ (through $\phi$)} if every pullback of an effective $\Delta$-system along $\phi$ is SFT covered.
\end{definition}

Plenty of examples of this phenomenon are known, for $\Gamma, \Gamma_1, \Gamma_2, \Gamma_3, \Delta$ infinite finitely-generated groups with decidable word problem, $\Delta$ non-amenable, and the maps $\phi$ the obvious ones:
\begin{itemize}
\item $\Z^{d+2}$ simulates $\Z^d$ \cite{Ho09} (the theorem above),
\item $\Z^d \rtimes \Gamma$ for $d \geq 2$ simulates $\Gamma$ for any semidirect product \cite{BaSa19},
\item $\Gamma_1 \times \Gamma_2 \times \Gamma_3$ simulates each of the $\Gamma_i$ \cite{Ba19},
\item $\Gamma_1 \times \Delta$ simulates (the non-amenable group) $\Delta$ \cite{BaSaSa25},
\item $\Z_2 \wr \Z$ simulates $\Z$ \cite{BaSa26}.
\end{itemize}

Simulation theorems have many consequences. For example, they typically allow the construction of strongly aperiodic subshifts of finite type (i.e.\ SFTs where every orbit is free). The construction of such SFTs is a common theme in the field, starting with the classical construction of Berger on $\Z^2$ \cite{Be66}. For several groups (including Thompson's $V$ \cite{BaSaSa26} and the Grigorchuk group \cite{Ba19}), the only known constructions of strongly aperiodic SFTs come from simulation theorems. On amenable groups, simulation theorems can also be used to obtain SFTs with arbitrary $\Pi^0_1$ entropies (this requires some control on the fibers in the simulation, but it does happen in the simulations above).

An interesting case is when a group simulates itself (through the identity map), called \emph{self-simulability}. In \cite{BaSaSa26}, it was shown that for finitely-generated groups with decidable word problem, self-simulability is equivalent to all effective subshifts on the group being sofic. The first examples were also given of self-simulable groups in \cite{BaSaSa26}. For example, the following groups are self-simulable
\begin{itemize}
\item $\Gamma \times \Delta$ where both $\Gamma, \Delta$ are f.g.\ non-amenable groups,
\item Thompson's $V$,
\item $\mathrm{GL}(n, \Z)$ for $n \geq 5$,
\item braid groups with at least $7$ braids,
\item all non-amenable branch groups,
\item certain right-angled Artin groups.
\end{itemize}

The case $\Gamma = \Delta = F_2$ of the first item is the prototypical example of a self-simulable group, and the other items are deduced from this. The results of the present paper generalize the first item by providing an almost complete characterization of the graph products of infinite groups that are self-simulable. In particular, we characterize self-simulable right-angled Artin groups. The results of this paper, however, do not follow from a direct application of the results in \cite{BaSaSa26} and require a new construction.

\subsection{New results on self-simulability of graph products}

See Section~\ref{sec:GraphProducts} for the precise definitions of a graph product of groups and a right-angled Artin group. Briefly, in a graph product of groups, we have a finite graph whose edges determined the commutation relations of subgroups assigned to the vertices, and a right-angled Artin group is a graph product $\Z$s. See Section~\ref{sec:SelfSimulability} for the definition of strong self-simulability; however, in the case of recursively presented groups, this coincides with the notion self-simulability.

We obtain the following criterion for strong self-simulability of graph products.

Let us say a set of nodes $C \subset V$ in a graph $(V, E)$ is \emph{disconnecting} if the number of connected components in $(V \setminus C, E \cap (V \setminus C)^2)$ is not equal to $1$. Thus, $C$ is disconnecting if either $C = V$, or there exist $a, b \notin C$ such that there is no path from $a$ to $b$ that does not enter $C$. (It may not seem natural to include the case $C = V$, but this is the more convenient choice for us.) 

\begin{theorem}
\label{th:GraphProducts}
Let $G = (V, E, (\Gamma_u)_{u \in V})$ be a finite graph every node of which is assigned a finitely-generated and infinite group $\Gamma_u$.  
Assume $G$ is not a clique with exactly one non-amenable node. Then the following conditions are equivalent:

\begin{enumerate}
\item $G$ has no disconnecting amenable clique.
\item The graph product $\Gamma(G)$ is strongly self-simulable.
\end{enumerate}

\end{theorem}

Note that the previously known examples of self-simulable groups stem from the simple case where $G$ is a clique on two non-amenable nodes.  

The proof of self-simulability of the groups generalizes the proof of self-simulability of $F_2 \times F_2$ by using the normal form for graph products to find suitable ``directions'' where information should be stored.

The case of a clique with a single non-amenable node of course is not characterized by the same condition. For instance, when $G$ has just one vertex $u$, whose vertex group $\Gamma_u$ is non-amenable, we have $\Gamma(G) \cong \Gamma_u$, and of course the self-simulability is not dictated by the graph (but instead by whether $\Gamma_u$ is self-simulable).

Even if we know the self-simulation status of the vertex groups $\Gamma_u$, giving a full characterization of self-simulability of graph products would require solving the following question from \cite{BaSaSa26}:

\begin{question}
\label{q:Product}
Let $\Gamma, \Delta$ be finitely-generated groups. If $\Gamma \times \Delta$ is (strongly) self-simulable and $\Delta$ is amenable, is $\Gamma$ necessarily (strongly) self-simulable?
\end{question}


The assumption that the node groups are infinite is also essential for the method, and the general case stays wide open.

As a corollary, we obtain a complete characterization of strongly self-simulable RAAGs.

The result in \cite{BaSaSa26} about RAAGs is the following:

\begin{theorem*}
Suppose $G = (V, E)$ is a finite connected graph which has two edges
with the property that no $v \in V$ is adjacent to both of them. Then the RAAG corresponding to the complement graph of $(V, E)$ is self-simulable.
\end{theorem*}

Question~9.8 of the same paper leaves open the full characterization of such RAAGs. In the case of RAAGs, we get a complete characterization as a consequence of Theorem~\ref{th:GraphProducts}.

\begin{theorem}
\label{th:RAAGs}
Let $G = (V, E)$ a finite graph. The following conditions are equivalent.
\begin{enumerate}
\item $G$ has no disconnecting clique. 
\item The right-angled Artin group $\Gamma_\Z(G)$ is self-simulable.
\end{enumerate}
\end{theorem}


Again, we recall that we consider a clique graph to be a disconnecting clique in itself.

It is shown in \cite{Wh81,Ta85} that the existence of a disconnecting clique can be determined in polynomial time. Thus, our theorem reduces the question of self-simulability to an easily checkable purely graph-theoretic condition.

\begin{example}
\label{ex:Cycle}
RAAGs of trees are never self-simulable as they are either cliques (with at most two nodes), or they are disconnected by any one of their inner vertices. RAAGs of cycles of length at least 4 are always self-simulable as any disconnecting vertex set must contain two non-adjacent vertices. Some more examples of applying this theorem can be found in Figure~\ref{fig:Examples}. \qee
\end{example}

\begin{figure}[htbp]
    \centering

    \begin{subfigure}[b]{0.3\textwidth}
        \centering
        \cycle{3}{0.5}
        \caption{$\Gamma_\Z(C_3) \cong \Z^3$ is amenable, so not self-simulable by \cite{BaSaSa26}. $C = C_3$ a disconnecting clique.}
        \label{fig:sub1}
    \end{subfigure}
    \hfill
    \begin{subfigure}[b]{0.3\textwidth}
        \centering
        \cycle{4}{0.5}
        \caption{$\Gamma_\Z(C_4) \cong F_2 \times F_2$ is self-simulable \cite{BaSaSa26}. There is no disconnecting clique.}
        \label{fig:sub2}
    \end{subfigure}    
    \hfill
    \begin{subfigure}[b]{0.3\textwidth}
        \centering
        \cycle{5}{0.5}
        \caption{$C_5$ has no disconnecting clique so the RAAG is self-simulable. (This case was open.)}
        \label{fig:sub3}
    \end{subfigure}
    \hfill
    \begin{subfigure}[b]{0.3\textwidth}
        \centering
        \begin{tikzpicture}[scale=0.5]
        \node[circle, inner sep=1.5pt] 
              (v1) at (0,0) {1};
        \node[circle, inner sep=1.5pt] 
              (v2) at (1.6,0) {2};
        \node[circle, inner sep=1.5pt] 
              (v3) at (3.2,0) {3};
	   \draw (v1) -- (v2) -- (v3);
        \end{tikzpicture}

        \caption{$\Gamma_\Z(P_3) \cong F_2 \times \Z$ is a product of a multi-ended group and an amenable group, so not self-simulable \cite{BaSaSa26}; $\{2\}$ is a disconnecting clique.}
        \label{fig:sub4}
    \end{subfigure}
    \hfill
    \begin{subfigure}[b]{0.3\textwidth}
        \centering
        \begin{tikzpicture}[scale=0.5]
        \node[circle, inner sep=1.5pt, fill] 
              (v1) at (0,0) {};
        \node[circle, inner sep=1.5pt, fill] 
              (v2) at (1.6,0) {};
        \node[circle, inner sep=1.5pt, fill] 
              (v3) at (3.2,0) {};
        \end{tikzpicture}

        \caption{$\Gamma_\Z(3K_1) \cong F_3$ is multi-ended, so not self-simulable \cite{BaSaSa26}, $\emptyset$ (or any vertex) is a disconnecting clique.}
        \label{fig:sub5}
    \end{subfigure}
    \hfill
        \begin{subfigure}[b]{0.3\textwidth}
        \centering
        \begin{tikzpicture}[scale=1]
	    \def\r{0.5}    
		\coordinate (C) at (1.5,0); 
	    \foreach \i in {1,...,4} {
	        \node[circle, fill=black, inner sep=1.5pt] 
	              (u\i) at ({90-(\i-1)*360/4}:\r) {};
	        \node[circle, fill=black, inner sep=1.5pt] 
	              (v\i) at ($(C)+({90-(\i-1)*360/4}:\r)$) {};
	    }
	    \foreach \i [evaluate=\i as \j using {int(mod(\i,4)+1)}] in {1,...,4} {
	        \draw (v\i) -- (v\j);
	        \draw (u\i) -- (u\j);
	    }
	    \draw (u2)--(v4);
		\end{tikzpicture}
        \caption{Disconnecting edge, so the RAAG $F_2^2 *_{\Z^2} F_2^2$ is not self-simulable. Follows also from Lemma~\ref{lem:Cut} above (from \cite{BaBlSaSa25}).}
        \label{fig:sub6}
    \end{subfigure}
        \hfill
    \begin{subfigure}[b]{0.3\textwidth}
        \centering
        \copathnumbers{6}{0.7}
        \caption{self-simulable by Corollary~8.13 in \cite{BaSaSa26} (in the complement $P_6$, $d(\{1,2\},\{5,6\}) > 2$). No disconnecting clique.}
        \label{fig:sub7}
    \end{subfigure}
    \hfill
    \begin{subfigure}[b]{0.3\textwidth}
        \centering
        \copathnumbers{5}{0.7}
        \caption{This is $\overline{P_5}$. $\{1,5\}$ is a disconnecting clique, so the RAAG $F_2^2 *_{\Z^2} \Z^3$ is not self-simulable.}
        \label{fig:sub8}
    \end{subfigure}
    \hfill
    \begin{subfigure}[b]{0.3\textwidth}
        \centering
        \cocycle{7}{0.5}
        \caption{This is $\overline{C_7}$. By an ad hoc proof \cite[Example 8.15]{BaSaSa26}, the RAAG is self-simulable. No  disconnecting clique.}
        \label{fig:sub9}
    \end{subfigure}
    \hfill
    \begin{subfigure}[b]{0.3\textwidth}
        \centering
        \cocycle{6}{0.5}
        \caption{This is $\overline{C_6}$. Results of \cite{BaSaSa26} do not apply. No  disconnecting clique.}
        \label{fig:sub10}
    \end{subfigure}
        \hfill
    \begin{subfigure}[b]{0.3\textwidth}
        \centering
        \begin{tikzpicture}[scale=0.5]
        \node[circle, inner sep=1.5pt] 
              (v1) at (0,0) {1};
        \node[circle, inner sep=1.5pt] 
              (v2) at (1.6,0) {2};
        \node[circle, inner sep=1.5pt] 
              (v3) at (3.2,0) {3};
        \node[circle, inner sep=1.5pt] 
              (v4) at (0.8,1.6) {4};
        \node[circle, inner sep=1.5pt] 
              (v5) at (2.4,1.6) {5};
        \node[circle, inner sep=1.5pt] 
              (v6) at (1.6,3.2) {6};
	   \draw (v1) -- (v2) -- (v3) -- (v5) -- (v6) -- (v4) -- (v5) -- (v2) -- (v4) -- (v1);
        \end{tikzpicture}

        \caption{In the second Sierpinski graph, $\{4,5,2\}$ is a disconnecting clique, so the RAAG is not self-simulable.}
        \label{fig:sub11}
    \end{subfigure}
    \hfill
    \begin{subfigure}[b]{0.3\textwidth}
        \centering
        \petersen
        \caption{In the Petersen graph, there is no disconnecting clique so the RAAG is self-simulable. Results of \cite{BaSaSa26} do not apply.}
        \label{fig:sub12}
    \end{subfigure}
\hfill

    \caption{Examples of applying Theorem~\ref{th:RAAGs} to various RAAGs.}
    \label{fig:Examples}
\end{figure}

One direction of Theorem~\ref{th:GraphProducts} is almost clear from Lemma~\ref{lem:Cut}. Indeed, we have:

\begin{corollary}
\label{cor:AmenableClique}
Let $G = (V,E,(\Gamma_v)_{v \in G})$ be a finite graph, $C \subset V$ a disconnecting clique in $G$. If for all $v \in C, \Gamma_v$ is amenable, then the graph product $\Gamma(G)$ is not self-simulable.
\end{corollary}

\begin{proof}
Let $G_1,G_2$ be two subgraphs of $G$ verifying that $G_1,G_2 \neq C, G_1 \cup G_2 = G$ and $G_1 \cap G_2 = C$. Then, we may write $\Gamma(G) = \Gamma(G_1) \ast_{\Gamma(C)} \Gamma(G_2)$. But since $C$ is a clique with amenable vertices, $\Gamma(C)$ is a direct product of amenable groups, and so it is amenable. We conclude with Lemma~\ref{lem:Cut} that $\Gamma(G)$ is not strongly self-simulable. 
\end{proof}

\subsection{Geometric consequences}

Groves and Hull proved in \cite{GrHu15} that a right-angled Artin group splits over an abelian subgroup if and only if its defining graph has a separating clique. Their proof can be generalized to amenable splittings (and possibly to general graph products), but it does not provide a geometric characterization. In particular, it does not yield that the property is quasi-isometry invariant.

In \cite{Za17}, Zaremsky showed that splitting over an abelian group (and, actually, the rank of this abelian group) is a commensurability invariant for right-angled Artin groups, and asked if it is a quasi-isometry invariant. Bensaid, Genevois and Tessera recently answered this in the positive in \cite{BeGeTe26} by using a geometric property, namely coarse separation by a family of subexponential growth. They prove more generally that a right-angled Artin group splits over an amenable subgroup if and only if it is coarsely separated by a family of subexponential growth.

Using a geometric property named extraterrestriality introduced in \cite{BaBlSaSa25}, we obtain that splitting over an amenable subgroup is a quasi-isometry invariant for non-trivial graph products. In particular, this provides another proof that splitting over an abelian subgroup is a quasi-isometry invariant for right-angled Artin groups. Extraterrestriality and coarse separation by a family of subexponential growth are independent properties, so we do not recover the main result of \cite{BeGeTe26}.

We say that a finitely-generated group $\Gamma$ \emph{splits over an amenable group} if it admits an action on a tree with no globally fixed vertices, no edge inversions, and all edge stabilizers amenable. This definition includes nontrivial HNN extensions over amenable groups, and nontrivial free products with amalgamation over an amenable subgroup.

\begin{corollary}
\label{cor:GraphProductGeometry}
Let $G = (V, E, (\Gamma_u)_{u \in V})$ be a finite graph every node of which is assigned a finitely-generated and infinite group $\Gamma_u$. Assume $G$ is not a clique with exactly one non-amenable node. Then the following conditions are equivalent:
\begin{enumerate}
\item $G$ has a disconnecting amenable clique.
\item The graph product $\Gamma(G)$ splits over an amenable subgroup.
\item The graph product $\Gamma(G)$ splits as an amalgamated free product over an amenable subgroup.
\item $G$ is extraterrestrial in the sense of \cite{BaBlSaSa25}.
\end{enumerate}
\end{corollary}

(See the end of this section for the assumption that $G$ is not a clique with exactly one non-amenable node.)

In particular, this corollary applies to right-angled Artin groups, and recovers the result of \cite{GrHu15}, since the amenable clique is of course an abelian clique in the case of a RAAG.

Note that Corollary~\ref{cor:GraphProductGeometry} also shows that for any non-trivial graph product that does not split as a direct product, a semi-splitting over an amenable subgroup (that is, an amenable subgroup relative to which the graph product has multiple ends) can be promoted to an algebraic splitting. This is somewhat reminiscent of the theorem of Stallings about ends of groups \cite{St68}.

Furthermore, as a corollary we obtain:

\begin{corollary}
Let $\Gamma(G)$ be a non-trivial graph product over finitely generated infinite groups that is not a clique with a single non-amenable vertex. Then the following are equivalent:
\begin{enumerate}
\item $\Gamma(G)$ is quasi-isometric to a group that splits over an amenable subgroup,
\item $\Gamma(G)$ splits over an amenable subgroup.
\end{enumerate}
\end{corollary}

\begin{proof}
This is a direct consequence of the fact that for such groups, splitting over an amenable subgroup is equivalent to being extraterrestrial, and of \cite[Theorem A]{BaBlSaSa25}, which states that extraterrestriality is a quasi-isometry invariant for finitely generated groups.
\end{proof}

The fourth item of Corollary~\ref{cor:GraphProductGeometry} means that the group admits ``UFOs'', in the terminology of \cite{BaBlSaSa25}. This is a technical quasi-isometry invariant notion, which intuitively says that we can locally disconnect the graph, so that nodes on the two sides can be put in one-to-one correspondence through the disconnecting set. This is related to the notion of coarse separation by a family of subgraphs of subexponential growth in \cite{BeGeTe26}, but the two notions seem to be incomparable.

Note that since any amenable split comes from a clique split, if the amenable groups among the $\Gamma_u$ have a property $P$ which is closed under direct products, then an amenable split exists if and only if a split with property $P$ exists. For example the case where $P$ is the class of virtually nilpotent group is related to the following:

\begin{conjecture*}[Conjecture~1.7 in \cite{BeGeTe26}]
Let $G = (V, E)$ be a finite graph and let the groups $\Gamma_u, u \in V$ be infinite finitely generated virtually nilpotent groups. Then $G$ has a disconnecting clique if and only if the graph product $\Gamma(G)$ is coarsely separable by a family of subexponential growth.
\end{conjecture*}

Namely specializing the corollary above, we have:

\begin{corollary}
Let $G = (V, E)$ be a finite graph and let the groups $\Gamma_u, u \in V$ be infinite finitely generated virtually nilpotent groups. Then $G$ has a disconnecting clique if and only if the graph product $\Gamma(G)$ is extraterrestrial.
\end{corollary}

As explained above, extraterrestriality can be seen as intuitively a form of ``separation by a thin set'', but it is not directly comparable with the notion sought in the conjecture above.

We note that in the case of a clique with exactly one non-amenable node, the characterization in Corollary~\ref{cor:GraphProductGeometry} does not continue to hold. For example, clearly if there is only one node $u$, and $\Gamma_u$ is nonamenable, then whether the corresponding graph product $\Gamma \cong \Gamma_u$ splits over an amenable subgroup is not a function of the shape of the graph. The situation is the same for cliques with a single non-amenable node, as the following proposition shows (see Section~\ref{sec:CliqueGeometryThing} for the proof).

\begin{proposition}
\label{prop:CliqueGeometryThing}
Let $G$ be a clique of finitely-generated groups, with a single non-amenable vertex $\Delta$. Then $\Gamma(G)$ splits non-trivially over an amenable subgroup if and only if $\Delta$ does. 
\end{proposition}

We also note an easy consequence of extraterrestriality that is related to amenable splittings and the work of Zaremsky in \cite{Za17}. Namely, he proves that a braid group over at least 4 braids is not commensurable to any group that splits over a free-group free subgroup. Similarly,

\begin{corollary}
For $n \geq 7$, the braid group $B_n$ is not quasi-isometric to any finitely generated group that splits over an amenable subgroup.
\end{corollary}

\begin{proof}
A finitely generated group that splits over an amenable subgroup is extraterrestrial \cite[Proposition 5.11]{BaBlSaSa25}. But extraterrestriality is a quasi-isometry invariant \cite[Theorem A]{BaBlSaSa25}, and $B_n$ is not extraterrestrial because it is self-simulable \cite{BaSaSa26}.
\end{proof}

\section{Definitions and convention}

\subsection{Strong self-simulability and obstructions}
\label{sec:SelfSimulability}

As noted in \cite{BaBlSaSa25}, it is better to use the following ``relativized'' definition of self-simulability. This has the benefit of removing assumptions on the word problem from our main theorems:

\begin{definition}
We say an action $(\Gamma, X)$ is \emph{relatively effective} if, letting $S$ be a generating set for $\Gamma$, there exists an effective action of the free group $(F_S, Z)$ such that $(\Gamma, X)$ is topologically conjugate to $(\Gamma, Z')$ where $Z'$ is the set of points in $Z$ stabilized by every relator of $\Gamma$. We say a group of \emph{strongly self-simulable} if all its relatively effective actions are SFT covered.
\end{definition}

When a group is recursively presented, relatively effective systems are effective, and thus strong self-simulability and self-simulability are the same notion.

It was shown in \cite{AuBaSa17} that self-simulability implies nonamenability and one-endedness for groups with decidable word problem, and \cite{BaSaSa26} shows that direct products of one-ended and amenable groups are also not self-simulable.

A geometric obstruction to self-simulation named \emph{UFOs} is presented in \cite{BaBlSaSa25}, which applies to a more general set of groups, and (when using the definition of strong self-simulation) also groups with arbitrarily complicated word problem. This obstruction amounts to a thin subset (typically, an amenable subgroup) separating the group. We make extensive use of these results, and in particular of the following lemma.

\begin{lemma}
\label{lem:Cut}
Let $\Gamma_1,\Gamma_2, \Delta$ be finitely generated groups with $\Delta$ amenable and $\iota_1 : \Delta \to \Gamma_1, \iota_2 : \Delta \to \Gamma_2$ proper embeddings. Then the amalgamated free product $\Gamma_1 \ast_{\iota_1,\iota_2} \Gamma_2$ is not strongly self-simulable.
 \end{lemma}

\subsection{Graph products}
\label{sec:GraphProducts}

A graph $(V, E)$ is always by default undirected and simple (no self-loops or multiple edges), and $V$ are the \emph{nodes} or \emph{vertices}. We also often assign groups to the vertices, and still call a triple $(V, E, (\Gamma_u)_{u \in V})$ a graph.

Let $(V, E)$ be a finite graph where to each $u \in V$ we have associated a group $\Gamma_u$. We assume that the $\Gamma_u$ are disjoint, apart from sharing the identity element. Then the corresponding \emph{graph product} is 
\[\Gamma(G) = \langle \bigcup_{u \in V} \Gamma_u \;|\; [\Gamma_u, \Gamma_v] \mbox{ when } (u, v) \in E \rangle, \]
i.e.\ this is can be thought of as the freest group where each $\Gamma_u$ ($u \in V$) embeds, and for each edge $(u, v) \in E$, the corresponding groups $\Gamma_u, \Gamma_v$ commute.

We note that the definition above makes sense even if the graph $(V, E)$ is infinite. However, we recall that the graph product is a finitely-generated group if and only if the graph $(V, E)$ is finite, and all the groups $\Gamma_u$ are finitely-generated. As explained in \cite[Section~3.2]{BaSaSa26}, an infinitely-generated group is never self-simulable. Thus, we restrict our attention to finite graphs.

We will sometimes confuse the vertices $V$ with the corresponding groups, and more generally sets of vertices with the subgroups generated by the corresponding vertex groups. For example, an \emph{amenable clique} is a clique in $(V, E)$ to every vertex of which is associated an amenable group. Note that a clique corresponds to a direct product, and a direct product of amenable groups is amenable, so indeed the subgroup generated by amenable vertex groups in a clique is itself amenable.

The \emph{right-angled Artin groups} or \emph{RAAGs} are the graph products where all the $\Gamma_u$ are infinite cyclic, that is $\Gamma_u \cong \Z$. The \emph{right-angled Coxeter groups} or \emph{RACGs} are the graph products where all the $\Gamma_u$ are cyclic of order 2, that is $\Gamma_u \cong \Z_2$. 

If the word problem is decidable in each $\Gamma_u$, then it is also decidable in the graph product. In particular, the word problem is decidable on RAAGs and RACGs.

\subsection{Convention}

\[ 0 \in \N \]

\section{Self-simulable graph products of infinite groups}

The goal of this section will be to prove Theorems~\ref{th:GraphProducts} and~\ref{th:RAAGs}. For the sake of brevity, we introduce the following definition.

\begin{definition}[Atomic Graph]

Let $(A_i)_{i \in I}, (B_j)_{j \in J}$ two families of f.g. groups such that the $A_i$ are infinite and amenable and the $B_j$ are non-amenable. Let $G = (V,E)$ a graph with $V = \{A_i	\;|\; i \in I\} \cup \{B_j \;|\; j \in J\}$. $G$ is atomic if:
\begin{itemize}
\item It admits no disconnecting clique the vertices of which are amenable.
\item $|J|\geq2$ or $G$ does not form a clique.
\end{itemize}
\end{definition}

The difficult implication in Theorem~\ref{th:GraphProducts} then reduces to showing that graph products of atomic graphs are strongly self-simulable. (The easy implication Corollary~\ref{cor:AmenableClique} proves directly.)

The following lemma states the main properties of atomic graphs.

\begin{lemma}
\label{lem:GraphTheory}
If $G$ is an atomic graph and $C \subset G$ is an amenable clique which does not contain all elements of $G$, then $G\setminus C$ is connected, every point of $C$ is connected to a point of $G\setminus C$ and $G\setminus C$ contains at least two points.
\end{lemma} 

\begin{proof}
  The connectedness of $G\setminus C$ is a direct consequence of the definition. If there exists a point $p \in C$ that is connected to no point of $G\setminus C$, then $C\setminus \{p\}$ is a disconnecting clique. Finally, if $|J|\geq2$, it is clear that $G\setminus C \supset \{B_j \;|\;j \in J\}$ contains at least two points, and if the $\{A_i \;|\; i \in I\}$ do not form a clique, then there must exist $A_{i_0} \notin C$. But every point of $C$ is connected to a point of $G\setminus C$, so that $A_{i_0}$ cannot be alone in $G\setminus C$ without $\{A_i \;|\; i \in I\}$ being a clique.
 \end{proof}

\subsection{RAAG warmup}

As a warmup, we begin with an outline of the argument in the case of RAAGs. Hence, let $G$ be a graph that has no disconnecting clique. The group $\Z$ has decidable word problem, so in this case strong self-simulability is the same as self-simulability. Thus, let $\Gamma(G) \curvearrowright X$ be an effective action. 
\begin{enumerate}
\item We construct a subshift of finite type that does the following. First, at every point $g$ of the group, we pick a set of directions $\bB(g)$, that is a set of vertices of $G$. The coset $g \langle \bB(g) \rangle $ is called the bush at $g$. We ask that the bush contains a coset that is isomorphic to $\Z^2$. We further ask that the bushes at $g$ and at $g s$ have a direction in common that commutes with $s$. Finally, we ask that the set of vertices corresponding to each bush forms a connected subgraph of $G$. In the case of RAAGs, simply choosing for $\bB(g)$ the complement of the possible last letters of a reduced writing of $g$ satisfies the conditions.

\item We then choose a direction in each bush on which to write an element of $\{0,1\}^\N$. The point now is to make sure that this element is in $X$, and that for every generator $s$, the element that is written on the bush at $g s$ is obtained by applying $s^{-1}$ to the element that is written on the bush at $g$. Indeed, if this is the case, the map that sends a configuration of the subshift to the element that is written on the bush at $1_{\Gamma_\Z (G)}$ will be equivariant.

\item Since the set of directions chosen at $g$ forms a connected graph, we can make sure that the word written on every direction beginning at $g$ is the same. This is done by synchronizing the word along the diagonals of each embedded $\Z^2$ in the bush. This is shown in Figure~\ref{fig:synchro}.

\begin{figure}
\begin{center}
\begin{tikzpicture}[scale=0.8]
\draw (-0.7071067811865475, 0.7071067811865475) edge[black, dotted, stealth-stealth, shorten <=2pt, shorten >=2pt] (0.2169304578186562, 0.9761870601839527);
\draw (-1.414213562373095, 1.414213562373095) edge[black, dotted, stealth-stealth, shorten <=2pt, shorten >=2pt] (0.4338609156373124, 1.9523741203679055);
\draw (-2.1213203435596424, 2.1213203435596424) edge[black, dotted, stealth-stealth, shorten <=2pt, shorten >=2pt] (0.6507913734559686, 2.928561180551858);
\draw (-2.82842712474619, 2.82842712474619) edge[black, dotted, stealth-stealth, shorten <=2pt, shorten >=2pt] (0.8677218312746248, 3.904748240735811);
\draw (0.2169304578186562, 0.9761870601839527) edge[black, dotted, stealth-stealth, shorten <=2pt, shorten >=2pt] (0.8944271909999159, 0.4472135954999579);
\draw (0.4338609156373124, 1.9523741203679055) edge[black, dotted, stealth-stealth, shorten <=2pt, shorten >=2pt] (1.7888543819998317, 0.8944271909999159);
\draw (0.6507913734559686, 2.928561180551858) edge[black, dotted, stealth-stealth, shorten <=2pt, shorten >=2pt] (2.6832815729997477, 1.3416407864998738);
\draw (0.8677218312746248, 3.904748240735811) edge[black, dotted, stealth-stealth, shorten <=2pt, shorten >=2pt] (3.5777087639996634, 1.7888543819998317);
\draw (0.8944271909999159, 0.4472135954999579) edge[black, dotted, stealth-stealth, shorten <=2pt, shorten >=2pt] (0.9701425001453318, -0.24253562503633294);
\draw (1.7888543819998317, 0.8944271909999159) edge[black, dotted, stealth-stealth, shorten <=2pt, shorten >=2pt] (1.9402850002906635, -0.4850712500726659);
\draw (2.6832815729997477, 1.3416407864998738) edge[black, dotted, stealth-stealth, shorten <=2pt, shorten >=2pt] (2.910427500435995, -0.7276068751089988);
\draw (3.5777087639996634, 1.7888543819998317) edge[black, dotted, stealth-stealth, shorten <=2pt, shorten >=2pt] (3.880570000581327, -0.9701425001453318);
\draw[thick] (-0.0, 0.0) -- (1.084652289093281, 4.880935300919764);
\draw[thick] (0.0, 0.0) -- (-3.5355339059327373, 3.5355339059327373);
\draw[] (-0.7071067811865475, 0.7071067811865475) -- (0.3775455079067336, 5.588042082106311);
\draw[] (0.2169304578186562, 0.9761870601839527) -- (-3.318603448114081, 4.51172096611669);
\draw[] (-1.414213562373095, 1.414213562373095) -- (-0.3295612732798139, 6.295148863292859);
\draw[] (0.4338609156373124, 1.9523741203679055) -- (-3.101672990295425, 5.487908026300643);
\draw[] (-2.1213203435596424, 2.1213203435596424) -- (-1.0366680544663613, 7.002255644479407);
\draw[] (0.6507913734559686, 2.928561180551858) -- (-2.8847425324767686, 6.464095086484596);
\draw[] (-2.82842712474619, 2.82842712474619) -- (-1.7437748356529088, 7.709362425665954);
\draw[] (0.8677218312746248, 3.904748240735811) -- (-2.6678120746581127, 7.440282146668548);
\draw[thick] (0.0, 0.0) -- (4.47213595499958, 2.23606797749979);
\draw[thick] (0.0, 0.0) -- (1.084652289093281, 4.880935300919764);
\draw[] (0.2169304578186562, 0.9761870601839527) -- (4.689066412818236, 3.2122550376837427);
\draw[] (0.8944271909999159, 0.4472135954999579) -- (1.9790794800931968, 5.328148896419722);
\draw[] (0.4338609156373124, 1.9523741203679055) -- (4.905996870636892, 4.1884420978676955);
\draw[] (1.7888543819998317, 0.8944271909999159) -- (2.8735066710931125, 5.7753624919196795);
\draw[] (0.6507913734559686, 2.928561180551858) -- (5.122927328455548, 5.164629158051648);
\draw[] (2.6832815729997477, 1.3416407864998738) -- (3.7679338620930287, 6.222576087419638);
\draw[] (0.8677218312746248, 3.904748240735811) -- (5.339857786274204, 6.140816218235601);
\draw[] (3.5777087639996634, 1.7888543819998317) -- (4.662361053092944, 6.669789682919595);
\draw[thick] (0.0, 0.0) -- (4.850712500726659, -1.2126781251816647);
\draw[thick] (0.0, -0.0) -- (4.47213595499958, 2.23606797749979);
\draw[] (0.8944271909999159, 0.4472135954999579) -- (5.745139691726575, -0.7654645296817069);
\draw[] (0.9701425001453318, -0.24253562503633294) -- (5.442278455144911, 1.993532352463457);
\draw[] (1.7888543819998317, 0.8944271909999159) -- (6.63956688272649, -0.3182509341817489);
\draw[] (1.9402850002906635, -0.4850712500726659) -- (6.412420955290243, 1.7509967274271239);
\draw[] (2.6832815729997477, 1.3416407864998738) -- (7.533994073726406, 0.1289626613182091);
\draw[] (2.910427500435995, -0.7276068751089988) -- (7.382563455435575, 1.508461102390791);
\draw[] (3.5777087639996634, 1.7888543819998317) -- (8.428421264726323, 0.576176256818167);
\draw[] (3.880570000581327, -0.9701425001453318) -- (8.352705955580907, 1.265925477354458);
\node () at (-0.8571067811865475, 0.5571067811865474) {$a$};
\node () at (0.01693045781865618, 0.7761870601839527) {$b$};
\node () at (0.6444271909999159, 0.4472135954999579) {$c$};
\node () at (0.8701425001453318, -0.34253562503633295) {$d$};
\end{tikzpicture}
\end{center}
\vspace{-0.5cm}
\caption{Suppose the bush at node $g$ is $\{a, b, c, d\}$, and $\{a, b\}, \{b, c\}, \{c, d\}$ commute pairwise. Each commuting pair spans a grid, and we use the dotted diagonal lines on the grids to synchronize $Y$-configurations stored on the rays $g\{s^n \;|\; n \in \N\}$ for different values of $s \in \{a, b, c, d\}$.}
\label{fig:synchro}
\end{figure}

\item Now, we use each bush $g \langle \bB(g) \rangle$ to simulate a Turing machine, by inscribing an instance of the tiling problem on a copy of $\Z^2$ that is contained in the bush. This Turing machine checks that the element of $\{0,1\}^\N$ that is written on $g \langle \bB(g) \rangle$ is in $X$, and that the element that is written on $g s \langle \bB(g s) \rangle$ is obtained by applying $s$. This can be done, because the bushes $g \langle \bB(g) \rangle$ and $g s \langle \bB(g s) \rangle$ share a direction, say $a$, that commutes with $s$. Hence, the cosets $g \langle a \rangle$ and $g s \langle a \rangle$ stay at bounded distance from one another.\footnote{A technical detail is that to be able to perform the comparison easily, we need to compute not only a single $x \in X$, but also its neighbors, see Definition~\ref{def:SetRepresentation}.} This is illustrated in Figure~\ref{fig:compute}.

\begin{figure}
\begin{center}
\begin{tikzpicture}[scale=0.8]
\draw[thick, black!30!white] (-0.0, 0.0) -- (1.084652289093281, 4.880935300919764);
\draw[thick, black!30!white] (0.0, 0.0) -- (-3.5355339059327373, 3.5355339059327373);
\draw[, black!30!white] (-0.7071067811865475, 0.7071067811865475) -- (0.3775455079067336, 5.588042082106311);
\draw[, black!30!white] (0.2169304578186562, 0.9761870601839527) -- (-3.318603448114081, 4.51172096611669);
\draw[, black!30!white] (-1.414213562373095, 1.414213562373095) -- (-0.3295612732798139, 6.295148863292859);
\draw[, black!30!white] (0.4338609156373124, 1.9523741203679055) -- (-3.101672990295425, 5.487908026300643);
\draw[, black!30!white] (-2.1213203435596424, 2.1213203435596424) -- (-1.0366680544663613, 7.002255644479407);
\draw[, black!30!white] (0.6507913734559686, 2.928561180551858) -- (-2.8847425324767686, 6.464095086484596);
\draw[, black!30!white] (-2.82842712474619, 2.82842712474619) -- (-1.7437748356529088, 7.709362425665954);
\draw[, black!30!white] (0.8677218312746248, 3.904748240735811) -- (-2.6678120746581127, 7.440282146668548);
\draw[thick, black!30!white] (0.0, 0.0) -- (4.47213595499958, 2.23606797749979);
\draw[thick, black!30!white] (0.0, 0.0) -- (1.084652289093281, 4.880935300919764);
\draw[, black!30!white] (0.2169304578186562, 0.9761870601839527) -- (4.689066412818236, 3.2122550376837427);
\draw[, black!30!white] (0.8944271909999159, 0.4472135954999579) -- (1.9790794800931968, 5.328148896419722);
\draw[, black!30!white] (0.4338609156373124, 1.9523741203679055) -- (4.905996870636892, 4.1884420978676955);
\draw[, black!30!white] (1.7888543819998317, 0.8944271909999159) -- (2.8735066710931125, 5.7753624919196795);
\draw[, black!30!white] (0.6507913734559686, 2.928561180551858) -- (5.122927328455548, 5.164629158051648);
\draw[, black!30!white] (2.6832815729997477, 1.3416407864998738) -- (3.7679338620930287, 6.222576087419638);
\draw[, black!30!white] (0.8677218312746248, 3.904748240735811) -- (5.339857786274204, 6.140816218235601);
\draw[, black!30!white] (3.5777087639996634, 1.7888543819998317) -- (4.662361053092944, 6.669789682919595);
\draw[thick, black] (0.0, 0.0) -- (4.850712500726659, -1.2126781251816647);
\draw[thick, black] (0.0, -0.0) -- (4.47213595499958, 2.23606797749979);
\draw[, black] (0.8944271909999159, 0.4472135954999579) -- (5.745139691726575, -0.7654645296817069);
\draw[, black] (0.9701425001453318, -0.24253562503633294) -- (5.442278455144911, 1.993532352463457);
\draw[, black] (1.7888543819998317, 0.8944271909999159) -- (6.63956688272649, -0.3182509341817489);
\draw[, black] (1.9402850002906635, -0.4850712500726659) -- (6.412420955290243, 1.7509967274271239);
\draw[, black] (2.6832815729997477, 1.3416407864998738) -- (7.533994073726406, 0.1289626613182091);
\draw[, black] (2.910427500435995, -0.7276068751089988) -- (7.382563455435575, 1.508461102390791);
\draw[, black] (3.5777087639996634, 1.7888543819998317) -- (8.428421264726323, 0.576176256818167);
\draw[, black] (3.880570000581327, -0.9701425001453318) -- (8.352705955580907, 1.265925477354458);
\node () at (0.1, -0.25) {$0q_0$};
\node () at (0.9944271909999158, 0.19721359549995793) {$1q_0$};
\node () at (1.8888543819998318, 0.6444271909999159) {$0$};
\node () at (2.7832815729997478, 1.0916407864998738) {$0$};
\node () at (3.6777087639996635, 1.5388543819998317) {$0$};
\node () at (1.0701425001453317, -0.49253562503633297) {$1$};
\node () at (1.9645696911452477, -0.045322029536375014) {$1$};
\node () at (2.8589968821451635, 0.4018915659635829) {$1q_0$};
\node () at (3.7534240731450796, 0.849105161463541) {$0$};
\node () at (4.647851264144995, 1.2963187569634989) {$0q_2$};
\node () at (2.0402850002906634, -0.7350712500726659) {$0$};
\node () at (2.9347121912905796, -0.28785765457270795) {$0$};
\node () at (3.8291393822904953, 0.15935594092724997) {$0$};
\node () at (4.723566573290411, 0.6065695364272079) {$0q_0$};
\node () at (5.617993764290326, 1.0537831319271658) {$1$};
\node () at (3.0104275004359953, -0.9776068751089988) {$1$};
\node () at (3.904854691435911, -0.5303932796090409) {$1$};
\node () at (4.799281882435826, -0.08317968410908294) {$1$};
\node () at (5.693709073435743, 0.36403391139087504) {$1$};
\node () at (6.588136264435658, 0.8112475068908329) {$1$};
\node () at (3.980570000581327, -1.2201425001453319) {$1$};
\node () at (4.874997191581243, -0.7729289046453738) {$1$};
\node () at (5.769424382581159, -0.3257153091454159) {$1$};
\node () at (6.663851573581074, 0.12149828635454207) {$1$};
\node () at (7.55827876458099, 0.5687118818545) {$1$};
\end{tikzpicture}
\end{center}
\vspace{-0.5cm}
\caption{One of the grids is used for computation, here we use $\{gc^md^n \;|\; m,n \in \N\}$ with $d$ as the ``space direction'' and $c$ as the ``time direction''. The example Turing machine shown is just an adding machine over a binary alphabet, in the construction we replace this by the machine that never halts if and only if the configuration is in $Y$.}
\label{fig:compute}
\end{figure}

\item Finally, we check that the map that sends a configuration to the element of $\{0,1\}^{\N}$ that is written at $1_{\Gamma_\Z (G)}$ is surjective onto $X$, and so it is a factor map. This is done by checking that the bushes can be chosen so that if $g \langle \bB(g) \rangle = h\langle \bB(h) \rangle$ and $\bB(g) = \bB(h)$, then $g=h$. Indeed, if this is the case then the bushes of different elements $g$ using the same set of directions are pairwise disjoint, so for each $g$, the bush at $g$ may be used entirely for the computation of $g$.
\end{enumerate}

We are now ready for the proof in the general case of a graph product. The main difference between a RAAG and a general amenable group is that (because the group might not contain any elements of infinite order) we need to replace the straight lines $\langle a \rangle$ by paths that are chosen by an SFT (this is indeed possible, using the fact that every infinite group admits a translation-like action of $\Z$ \cite{Se14}).

In the amenable case, when we use a direction at $g$, we will consider the entire bush (the coset $g\bB(g)$) ``used'', and ensure that other elements $g$ have essentially disjoint bushes (more precisely, we will allow a bounded number of reuses). In the case of a nonamenable group, a paradoxical decomposition of a group allows every node to have a separate path. We can use an SFT to mark a paradoxical decomposition as in \cite{BaSaSa26}.

Finally, when the groups do not necessarily have decidable word problem, we need to work with a relative action instead of an actual action. It turns out that this does not really make a difference in the proofs.

\subsection{Path subshifts}

Let us assume $G = (V,E,(\Gamma_u)_{u \in V})$ where $V = I \sqcup J$, with $\Gamma_i$ amenable for every $i \in I$, and $\Gamma_j$ non-amenable for every $j \in J$. For $v \in V$, denote $\lk(v)$ its link, i.e. $\lk(v) = \{u \in V \;|\; \{u,v\} \in E\}$. Note that $v \notin \lk(v)$.

Let $K_v \Subset \Gamma_v$. Define for all $i \in I, \Sigma_i = K_i^2 \times \{\bb\}$, for all $j \in J, \Sigma_j = K_j^3 \times \{\br,\bc\}$ and $\Sigma = (\prod_{i \in I} \Sigma_i) \times (\prod_{j \in J} \Sigma_j)$.  For each $u \in V$, write $\col : \Sigma_u \to \{\bb,\br,\bc \}$ the projection on the last coordinate, and $\overline{\symb{t}}$ the opposite color of color $\symb{t}$ (that is $\overline{\br} = \bc, \overline{\bc} = \br$ and $\overline{\bb} = \bb$). Let $S_v \Subset \Gamma_v$ finite generating sets of the $\Gamma_v$. Let $S = \bigcup_{v \in v} S_v$. 

Define the path subshift $\cP$ on alphabet $\Sigma$ by demanding the following conditions of every $((\rho_i = ((\ell_i, r_i),\bb)_{i \in I},(\rho_j = (({\ell_j}^{\br} ,{\ell_j}^{\bc} ,r_j),c_j))_{j \in J}) \in \cP$ at every $g \in \Gamma$.
\begin{enumerate}
	\item For every $i \in I, r_i (g \ell_i (g)) = \ell_i (g)^{-1}$.
	\item For every $i \in I, \ell_i ( g r_i (g)) = r_i (g)^{-1}$.
	\item For every $j \in J, c_j (g {\ell_j}^{\br} (g))) = \br$ and $r_j (g {\ell_j}^{\br}(g)) = {\ell_j}^{\br} (g)^{-1}$.
	\item For every $j \in J, c_j (g {\ell_j}^{\bc}(g))) = \bc$ and $r_j (g {\ell_j}^{\bc}(g)) = {\ell_j}^{\bc} (g)^{-1}$.
	\item For every $j \in J, \ell_j^{c_j (g)}(g r_j (g)) = r_j (g)^{-1}.$
	\item For every $\{u,v\} \in E$, for every $a \in S_v, \rho_u (g) = \rho_u (g a)$.
\end{enumerate}
It is clear that $\cP$ is of finite type.

For $\rho \in \cP$ and $v \in V$, we define a path by following the left edges of the opposite color of a vertex:
\[ \gamma_g^v (n,\rho) = \begin{cases}
1_\Gamma &\mbox{ if } n=0\\
\gamma_g^v (n-1,\rho) \ell_v ^{\overline{c_v(g)}}(g \gamma_g^v (n-1,\rho)) &\mbox{ otherwise}. \end{cases}\] 

Note that for every $h \in \Gamma_v$, \[\gamma_{h^{-1}}^v (n,\rho) = \gamma_{1_\Gamma}^v (n,h\rho).\]
This is a simple calculation, but one may also avoid the calculation with the correct mental picture: If we visualize the configurations on the right Cayley graph, the shift by $h$ moves the vertex $1_\Gamma$ to $h$, and carries the configuration by the unique graph automorphism (when we include the generators as edge labels). The equality comes from that fact that the definition of $\gamma_g^v (n,\rho)$ can be seen in terms of local movement of a ``reading head'' on the configuration from initial position $g$. The configuration $(n,\rho)$ relative to position $h^{-1}$ and the configuration $(n,h\rho)$ relative to position $1_\Gamma$ are the same, by the definition of the shift map.

Note also that for every $u \in \lk(v), h \in G_u$,  \[\gamma_{1_\Gamma}^v (n,\rho) =  \gamma_{1_\Gamma}^v (n,h\rho).\]
This in turn follows by an easy induction from the last item of the definition of the subshift.

\begin{claim}
	\label{cl:PathSubshift}
	For a suitable choice of the $K_v \Subset \Gamma_v$, the path subshift contains a configuration $\rho$ that is such that for every $v \in V$, and for every $g \in \Gamma$, the path $(n \in \N \mapsto \gamma_g^v(n,\rho))$ is injective.
\end{claim}
\begin{proof}
	For every $v \in V$, denote $\pi_v$ the canonical projection $\Gamma(G) \to \Gamma_v$ defined by mapping $\pi_v(g) = g$ for $g \in \Gamma_v$, $\pi_v(g) = 1_{\Gamma(G)}$ for $g \in \Gamma_u$ when $u \neq v$, and extending uniquely. This gives a well-defined homomorphism, and of course the restriction $\pi_v|_{\Gamma_v} : \Gamma_v \to \Gamma_v$ is the identity map.
	
	Seward showed in \cite[Theorem 4.1]{Se14} that every finitely generated infinite group $\Gamma$ has a translation like action of $\Z$, i.e.\ an action $\ast$ that is free and such that for all $t^n \in \Z = \langle t \rangle$, the set $\{g^{-1}(g \ast t^n) \;|\; g \in \Gamma\}$ is finite.
	
	 In particular, for every $i \in I$ there is a translation like action $\ast_i$ of $\Z$ on $\Gamma_i$.

	Now, define $K_i = \{g^{-1}(g \ast_i t) \;|\; g \in \Gamma_i\}$ and define $\rho_i|_{\Gamma_i}$ by
	 \[\forall g \in \Gamma_i, \rho_i(g) = (g^{-1} (g \ast_i t), g^{-1} (g \ast_i t^{-1})),\bb).\]
	 
	Then, for all $n \in \N$, for all $g \in \Gamma_i$, we have $g \gamma_g^i(n,\rho) = g \ast_i t^n $. Since the action is free, the path is injective. Extend $\rho_i|_{\Gamma_i}$ by the trivial extension, that is, define $\rho_i(g) = \rho_i|_{\Gamma_i}(\pi_{A_i}(g))$ and now, for all $g \in \Gamma$, the paths $g \gamma_g^i(n,\rho) = g \ast_i t^n$ are also injective.
	
	Now, it was proven in \cite{BaSaSa26} that for every  $j \in J$, there exists $K_j \Subset \Gamma_j$ such that the paradoxical subshift $\rho_j|_{\Gamma_j}$ on $\Gamma_j$ is non-empty. Now define $\rho_j$ as the trivial extension of any configuration of the paradoxical subshift on $\Gamma_j$, that is, define $\rho_j(g) = \rho_j|_{\Gamma_j}(\pi_{B_j}(g))$ for every $j \in J$. Since the map $(g,n)\in \Gamma_j \times \N \mapsto g\gamma_g^j(n+1,\rho)$ is already injective, it follows that the path $n \in \N \mapsto \gamma_g^j(n,\rho)$ is injective. 
	
	It is straightforward to check that $\rho \in \cP$.
\end{proof}

The previous claim shows that it is possible to find infinite paths beginning at every point of the group in every direction. In the amenable directions, these paths are pairwise disjoint, but two distinct elements of $\Gamma(G)$ may have the same path. By contrast, in the non-amenable directions, two distinct elements $g,h$ of $\Gamma(G)$ have disjoint paths (if we discount the elements $g, h$ themselves). This is why two colors are necessary in the non-amenable directions - one color is used for the root of the path and another for the rest.

In general, the proof is adapted from the special case of RAAGs (whose outline was given in the beginning of the section) to general graph products in the following way. At each $g$, we choose a set of vertices $\bB(g)$. We write $\mathcal{B}(g)$ the set of elements that can be reached by beginning at $g$ and following a path in one of the directions of $\bB(g)$. The point is that we want that if the intersection $\mathcal{B}(g) \cap \mathcal{B}(h)$ is non-empty and $\bB(g) = \bB(h)$, then $g=h$. If this is the case, then by having as many layers of computation as there are possible subgraphs $\bB(g)$, $g$ can use the entire set $\mathcal{B}(g)$ to do its computation

We have shown that we can choose disjoint paths at every point in the non-amenable directions, and so the non-amenable vertices may be used in $\bB(h)$ for every $h$ without overlap. In the case of an amenable vertex $A_i$, we will need to be more careful. The idea in this case is that we only include $A_i$ in $\bB(g)$ if $h$ cannot be written in reduced form so that it ends with an element of $A_i$. Namely, in this case it happens that $h$ is the \emph{only} element of $\mathcal{B}(h)$ that cannot be written in reduced form so that it ends with an element of $\bB(h) \cap \{A_i \;|\; i \in I\}$. We explain this in detail in Claim~\ref{claim:NonEmpty}.

\subsection{Bushes}

  A colored edge is an ordered pair $((u,c_u),(v,c_v))$ where $\{u,v\} \in E, c_u,c_v \in \{\br,\bc,\bb\}$. Denote by $\cE$ the set of colored edges. A colored vertex is an element of $V \times \{\br,\bc,\bb\}$. Denote by $\cV$ the set of colored vertices.
  
  Any function defined on edges (respectively vertices) is extended trivially to colored edges (respectively colored vertices) by ignoring the color. Let also $\openbox$ be a blank symbol. Let $A = \Sigma \times 2^V \times 2^{\cE} \times 2^{\cV \times 2^V} \times (\Omega \cup \{ \openbox \} )^{E \times 2^V}$, and define the bush subshift $\mathcal S$ on alphabet $A$ by demanding that the following conditions hold for every $\bS = (\rho,\bB,\bD,\bI,\bL) \in \mathcal S$ at every $g \in \Gamma$.
\begin{enumerate}
	\item $\rho \in \cP$.
	\item $\bB(g)$ contains at least two nodes and the induced subgraph of $(V, E)$ on these nodes is connected.
	\item For every $v \in V$, for every $a \in S_v, \bB(g a) \cap \bB(g) \cap \lk(v) \neq \emptyset$.
	\item If $u,v \in \bB(g)$ and $\{u,v\} \in E$, then $((u,\overline{c_u (g)}),(v,\overline{c_v (g)})) \in \bD(g)$.
	\item If $((u,c_u),(v,c_v)) \in \bD(g)$, then $((u,c_u),(v,c_v)) \in \bD(g \ell_u^{c_u} (g))$ and\\
	 $((u,c_u),(v,c_v)) \in \bD(g \ell_v^{c_v} (g))$.
	\item If $u \in \bB (g)$, then $((u,\overline{c_u (g)}),\bB(g)) \in \bI(g)$.
	\item If $((u,c_u),C) \in \bI(g)$ then $((u,c_u),C) \in \bI(g \ell_u^{c_u} (g))$.
	\item If $((u,c_u),(v,c_v)) \in \bD(g)$, then for any $C \subset V$ such that $\{u,v\} \subset C, \bL(g \ell_u^{c_u} (g))(\{u,v\},C) = \bL(g \ell_v^{c_v} (g))(\{u,v\},C)$.
	\item If $((u,c_u),C) \in \bI(g)$ then for any  $u' \in \lk(u)$, for any $a \in S_{u'}$, for any $((u,c_u),C') \in \bI(g a)$ and for any $\{u,v\} \subset C, \{u,v'\} \subset C'$,\\
	$\bL(g a)(\{u,v'\},C')(1_\Gamma) = \bL(g)(\{u,v\},C)(a^{-1})$.
\end{enumerate}
The role of the bush subshift is the following. $\bB$ chooses a bush, that is a set of directions at every vertex $g$, which is supposed to satisfy that if $g \langle \bB(g) \rangle = h \langle \bB(h) \rangle$ and $\bB(g) = \bB(h)$ then $g =h$. Hence, the layer of index $\bB(g)$ of the bush subshift on the subset $g \langle \bB(g) \rangle$ may be used entirely by $g$ for computation without interference.

$\bD(g)$ identifies the planes that are entirely contained in $\bB(g)$ on which we will later be able to embed the Wang tiles that will do the actual computation (rule 2 ensures that there will be at least one plane on which to do computation). 

$\bI$ identifies the edges of $g \langle \bB(g)\rangle$, that is the paths where it stays at bounded distance from another bush (if $u$ commutes with some group $G_v$, then the paths $g \gamma_g^v (n,\rho)$ and $g u \gamma_{gu}^v(n,\rho)$ always stay close). By rule 3, we know that the bushes corresponding to two adjacent elements will always share an edge, so that synchronization does happen.  

Finally, $\bL$ contains the layer on which configurations of $\Gamma \curvearrowright X$ will be stored. By synchronizing $\bL$ along diagonals of any plane contained in the bush (rule 8), we ensure that the same configuration of $\Gamma \curvearrowright X$ is written along any edge of $g \langle \bB(g)\rangle$. Rule 9 ensures that the bushes that have two paths that stay close are synchronized.


For proving that $\mathcal{S}$ is non-empty, we will use some basic results about words in graph products.

A \emph{writing} of an element $g \in \Gamma(G)$ is a sequence $g_1, \cdots, g_n$ such that each $g_i$ belongs to some $\Gamma_v$ and $g = g_1\dots g_n$. The elements $g_i$ of the sequence are called \emph{syllables}.

Clearly, the element represented by the writing $g_1\dots g_n$ is not modified by permuting $g_i$ and $g_{i+1}$ if $g_i \in \Gamma_u, g_{i+1} \in \Gamma_{v}$ and $\{u,v\} \in E$. Similarly, the element is not modified by replacing $g_i, g_{i+1}$ by their product $g_i g_{i+1}$ if $g_i, g_{i+1} \in \Gamma_v$. Finally, the element is not modified by deleting $g_i$ if $g_i = 1_{\Gamma_v}$.

We say that a writing $g = g_1 \dots g_n$ is \emph{graphically reduced} if it cannot be shortened by applying the three operations above.

\begin{definition}
If $g \in \Gamma(G)$, we define $\tail(g)$ as the set of vertices $v \in V$ such that there exists a graphically reduced writing of $g$ ending with an element of $\Gamma_v$.
\end{definition}

\begin{lemma}
The set $\tail(g)$ is a clique.
\end{lemma}

\begin{proof}
Suppose $v,v' \in \tail(g)$. Let $w$ be a writing of $g$ ending with an element of $\Gamma_v$, and $w'$ one ending with an element of $\Gamma_{v'}$. Then one can turn $w$ into $w'$ by swapping adjacent commuting group elements in the writing \cite{HeMe95}. In other words, there is a sequence of writings $w_0, w_1, \ldots, w_k$ of $g$ where for all $i$ we can write $w_i = gabh, w_{i+1} = gbah$ where $a \in \Gamma_u, b \in \Gamma_{u'}$ with $(u, u') \in E$.

In particular, the rightmost syllable of $w_0$ corresponds some syllable in $w_k$, and it was moved there by pairwise swaps of elements coming from commuting groups. At some point, it thus had to swap with the rightmost syllable of $w_k$. This means $(v, v') \in E$.
\end{proof}

In the case of a RAAG, the clique $\tail(g)$ corresponds to the rightmost clique in the normal form described in \cite{Va94}.
  




\begin{claim}
\label{claim:NonEmpty}
If $G$ is atomic, then the $\Gamma(G)$-subshift $\mathcal S$ is a non-empty SFT.
\end{claim}

\begin{proof}
  It is clear by the definition that $\mathcal S$ is of finite type. Let $\rho \in \cP$ as in Claim~\ref{cl:PathSubshift}, i.e.\ such that the paths along every direction are injective. Let $x \in X$, and for all $g \in \Gamma, y^{(g)} \in Y$ such that $(y_n^{(g)}(1_\Gamma))_{n \in \N} = g^{-1} x$. 
  
  
   Then define $\bB \in {2^V}^\Gamma$ by $\bB(g) = (V\setminus \tail(g)) \cup \{B_j | j \in J\}$. 
  
  Define $\bD(g)$ by $\forall ((u,c_u),(v,c_v)) \in \cE, ((u,c_u),(v,c_v)) \in \bD(g)$ if and only if $\exists g_0 \in \Gamma,n,m \in \N$ such that $\{u,v\} \subset \bB(g_0), c_u = \overline{c_u (g_0)}, c_v = \overline{c_v (g_0)}$ and $g = g_0 \gamma_{g_0}^u (n,\rho) \gamma_{g_0 \gamma_{g_0}^u (n,\rho)}^v (m,\rho)$. 

Define $\bI$ by $\forall ((u,c_u),C) \in \cV \times 2^V, ((u,c_u),C) \in \bI(g)$ if and only if $\exists g_0 \in \Gamma, n \in \N, c_u = \overline{c_u (g_0)}$ such that $C = \bB(g)$ and $g = g_0 \gamma_{g_0}^v (n,\rho)$. 
  
  Define $\bL$ by $\bL (g \gamma_g^u (n,\rho) \gamma_{g \gamma_g^u (n,\rho)}^v (m,\rho))(\{u,v\},\bB(g)) = y^{(g)}_{n+m-1}$ for $\{u,v\} \in E, \{u,v\} \subset \bB(g), n+m\geq1$ and $\bL(h)(\{u,v\},C) = \openbox$ everywhere else.   
  
  Note that $\bL$ is well-defined, because if there exists $g \gamma_g^u (n,\rho)\gamma_{g \gamma_g^u (n,\rho)}^v (m,\rho) = g' \gamma_{g'}^{u'} (n',\rho)\gamma_{g' \gamma_{g'}^{u'} (n,\rho)}^v (m,\rho), \bB(g) = \bB(g')$ and $\{u,v\} = \{u',v'\}$, then since the bushes $g \langle \bB(g) \cap \{A_i \;|\; i \in I\} \rangle $ and $g' \langle \bB(g) \cap \{A_i \;|\; i \in I\} \rangle$ only have one point the last clique of which contains no element of $\bB(g) \cap \{A_i\;|\; i \in I\}$ (respectively $\bB(g') \cap \{A_i\;|\; i \in I\}$), it must be that $\Gamma_u,\Gamma_v,\Gamma_u'$ and $\Gamma_v'$ are non-amenable. But then the injectivity of $(n,g) \mapsto g \gamma_g^j (n,\rho)$ yields that $g = g'$ and it follows that $n=n', m = m'$, because the paths $n \in \N \mapsto \gamma_g^u (n,\rho)$ are injective by Claim~\ref{cl:PathSubshift}. 
  
 We argue that $(\rho,\bB,\bD,\bI,\bL) \in \mathcal S$.
\begin{enumerate}
  	\item This is part of the definition.
  	\item $\bB(g) = G \setminus (\tail(g) \cap \{G_i \;|\; i \in I\})$ is the complement of an amenable clique, so by Lemma~\ref{lem:GraphTheory}, $\bB(g)$ contains at least two nodes and is connected.
  	\item Let $g \in \Gamma,v \in V, a \in S_v$. There are two cases. If $v \in \bB(g)$, then since by the previous point $\bB(g)$ is connected and contains at least two nodes, there exists $u \in \bB(g) \cap \lk(v), u\neq v$. But then either $\Gamma_u$ is non-amenable, and so $u \in \bB(g) \cap \bB(g a) \cap \lk(v)$, or $u \notin \tail(g)$. In the latter case since $v \neq u$, it follows that $u \notin \tail(g a)$, and $u \in \bB(g) \cap \bB(g a) \cap \lk(v)$.
	
	 Otherwise if $v \notin \bB(g)$, then by Lemma~\ref{lem:GraphTheory}, there is $u \in \bB(g) \cap \lk(v)$. But $\tail(g a) \subset \tail(g) \cup \{v\}$, and therefore either $u \notin \tail(g a)$ or $u$ is non-amenable, so that $u \in \bB(ga)$. 
  	\item This follows from the definition of $\bD$ with $n = m = 0$.
  	\item This follows from the definition of $\bD$ and the fact that $\ell_u^{c_u} (g)$ and $\ell_v^{c_v} (g_0)$ commute.
  	\item This follows from the definition of $\bI$ with $n=0.$
  	\item This follows from the definition of $\bI$ with $g = g_0 \gamma_g^v (n,\rho)$ and $h = g_0 \gamma_{g_0}^v (n+1,\rho)$.
	 \item This follows from the definition of $\bL$ and the fact that $\bD(g)$ is empty on elements not of the form $g_0 \gamma_{g_0}^u (n,\rho) \gamma_{g_0 \gamma_{g_0}^u (n,\rho)}^v (m,\rho)$. Indeed, if $h = g_0 \gamma_{g_0}^u (n,\rho) \gamma_{g_0 \gamma_{g_0}^u (n,\rho)}^v (m,\rho)$ with $\{u,v\} \subset \bB(g_0)$, then the condition is verified at $\bL (h)(\{u,v\},\bB(g_0))$, and since $\bL (h)(\{u,v\},C) = \openbox$ whenever $C \neq \bB(g_0)$, the condition is also verified in this case.
  	\item Assume $g = g_0 \gamma_{g_0}^u (n,\rho)$ with $c_u = \overline{c_u (g_0)}$ and $C = \bB(g_0)$. Then by definition of $\bL$, we have $\bL(g)(\{u,v\},C)(a) = (a^{-1}g_0^{-1} x)_{n-1}$. Since $v$ is adjacent to $u$, it follows from commutativity that $g a = g_0 a \gamma_{g_0}^u (n,\rho) = g_0 a \gamma_{g_0 a}^u (n,\rho)$ by rule 6 of the path subshift. Finally, since $v$ and $u$ are adjacent, $c_u (g_0 a) = c_u (g_0) = \overline{c_u}$. 
	
	Hence, the definition of $\bL$ yields $\bL(g_0 a)(\{u,v'\},C')(1_\Gamma) = y^{(g_0 a)}_{n-1} (1_\Gamma) = (a^{-1}g_0^{-1} x)_{n-1}$. \qedhere
 \end{enumerate}
\end{proof}

 The previous claim shows that it is possible to stitch a bush at every element of $\Gamma$. These bushes will next be used for computation, to ensure that the configurations $\Omega^\omega$ on the paths starting at each $g$ are indeed in $Y$. Then, as we clarify below, item~8 and item~9 above already guarantee that every configuration encodes the orbit of a point in $X$. 

\subsection{(Relatively) effective actions and set representations}

We recall the set representation of an effective action from \cite{BaSaSa26}. First, consider the case where all vertex groups have decidable word problem. Let $\Gamma \curvearrowright X \subset \{0,1\}^\omega$ be an effectively closed action. We recall the notion of its set representation from \cite{BaSaSa26}:

\begin{definition}
\label{def:SetRepresentation}
If $S \Subset \Gamma$ is a finite generating set for the group $\Gamma$ and $\Gamma \curvearrowright X \subset \{0,1\}^\omega$ is an action, the corresponding \emph{set representation} of the action is
\[ \{ x \in (\{0,1\}^S)^\omega \;|\; (x_n(1_\Gamma))_{n \in \N} \in X \mbox{ and } \forall s \in S, (x_n(s))_{n \in \N} = s (x_n(1_\Gamma))_{n \in \N}\} \]
\end{definition}

\begin{lemma}[\cite{BaSaSa26}, Remark 2.7]
An action is effective if and only if its set representation is effectively closed for some generating set, in which case it is effectively closed for all generating sets.
\end{lemma}

Let now $Y \subset \Omega^\omega$ be the set representation of $\Gamma \curvearrowright X$ for the generating set $S$. Recall that $\Omega = \{0,1\}^S$ and $Y = \{y \in \Omega^\omega \;|\; (y_n (1_\Gamma))_{n \in \N} \in X \mbox{ and } \forall s \in S, (y_n (s))_{n \in \N} = s (y_n (1_\Gamma))_{n \in \N}\}$.

Since $\Gamma \curvearrowright Y$ 
is effectively closed, there exists a Turing machine $\mathcal M$ that recognizes the forbidden patterns of $Y$. Alternatively, the computation of $\mathcal M$ on input $y \in (\Omega^S)^\omega$ terminates if and only if $y \notin Y$.

When $\Gamma$ does not have decidable word problem, we will instead have an effective action of the free group $(F_S, Z)$ with free generators $S$ on some $Z \subset \{0,1\}^\Z$ (in other words, $Z$ is effectively closed, and to each $S$ we associate a self-homeomorphism of $Z$), and $(\Gamma, X) \cong (\Gamma, Z')$ where $Z' \subset Z \subset \{0,1\}^\omega$ is the subset where the action of $\Gamma$ is well-defined, i.e.\ all relators of $\Gamma$ (w.r.t.\ the generating set $S$) act trivially.

Since the group $F_S$ certainly has decidable word problem, we can apply the set representation to this case. This gives the following lemma.

\begin{lemma}
\label{lem:RelativeSetRepresentation}
If $(\Gamma, Z)$ is a relatively effective system, there is an effectively closed action of $F_S$ on $X \subset \{0,1\}^\omega$ such that, letting $Y \subset (\{0,1\}^S)^\omega$ be the set representation of that action and $\pi : (\{0,1\}^S)^\omega \to (\{0,1\}^\omega)^S$ the natural bijection (transposition), we have that
\[ \{ z \in (\{0,1\}^\omega)^\Gamma \;|\; \forall g \in \Gamma: \pi(gz|_S) \in Y \} \]
under the shift action of $\Gamma$ factors onto $(\Gamma, Z)$ by $z \mapsto z|_{\textrm{id}_\Gamma}$.
\end{lemma} 

Thus, the proof of the main theorem amounts to checking that we can assign a point from $X$ to each $g \in \Gamma$, and to verify the consistency requirement at each $gz|_S \sim z|{g^{-1}S}$ (where $\sim$ denotes that the patterns are in canonical bijection). Due to our shift convention, if $z$ is encoded at the identity element $\textrm{id}$, the point $gz$ should be encoded at the element $g^{-1}$.

\subsection{The computation subshift}

Let $(F_S, X)$ be the system from the previous section, so that $(F_S, X')$ (seen as a $\Gamma$-system) factors onto $(\Gamma, Z)$. Let $Y \subset (\{0,1\}^S)^\omega$ be the set representation of $(F_S, X)$. Alternatively, if $\Gamma$ has decidable word problem, one may take $Y$ directly the set representation of a $\Gamma$-system $(\Gamma, X)$. 

For $W$ a Wang tileset containing a specific symbol $\bseed \in W$, consider the map $\eta$ which associates to any valid tiling of the quarter plane $\Z^2$ with symbol $\bseed$ at $(0,0)$ the contents of the bottom row starting at $(0,0)$. That is, $\eta(\tau) = \tau|_{\{(n,0) \;|\; n \geq 1\}}$. 

It is known that there exists a tileset $W$ such that $\eta$ surjects valid tilings of $\Z^2$ with $\bseed$ at $(0,0)$ onto inputs on which $\mathcal M$ does not terminate \cite{Wa60}. In other words, if we write
\[ Y = \{ y \in (\{0,1\}^S)^\omega \;|\; (y_n(1_{F_S}))_{n \in \N} \in X \mbox{ and } \forall s \in S, (y_n(s))_{n \in \N} = s (y_n(1_{F_S}))_{n \in \N}\}, \]
then we have $W \supset \Omega$ and
\[ \eta (\{\mbox{valid tilings }\tau : \Z^2 \to W \;|\; \tau(0,0) = \bseed \}) = Y. \]

 Now define the computation subshift $\mathcal Z$ on alphabet $\Sigma \times 2^V \times 2^{\cE} \times 2^{\cV \times 2^V} \times (\Omega \cup \{ \openbox \})^{E \times 2^V} \times 2^{\cE} \times W^{\cE}$ by demanding the following conditions of every \ $\bS = (\rho,\bB,\bD,\bI,\bL,\bP,\bT) \in \mathcal Z$ at every $g \in \Gamma$. 
\begin{enumerate}
	\item $(\rho,\bB,\bD,\bI,\bL) \in \mathcal S$.
	\item $\exists ! e_g = \{u,v\} \subset E \cap \bB(g), ((u,\overline{c_u (g)}),(v,\overline{c_v (g)}))\in \bP(g)$.
	\item If $e = ((u,c_u),(v,c_v)) \in \bP(g)$ then $e \in \bP(g \ell_u^{c_u} (g))$ and $e \in \bP(g \ell_v^{c_v} (g))$. 
	\item $\bT(g)(((u,\overline{c_u (g)}),(v,\overline{c_v (g)}))) = \bseed.$
	\item If $e = ((u,c_u),(v,c_v)) \in \bP(g)$ and $((u,c_u),C) \in \bI(g)$ and $\bT(g)(e) \neq \bseed$ then $\bL(g)(\{u,v\},C) = \bT(g)(e)$. 
	\item If $e = ((u,c_u),(v,c_v)) \in \bP(g)$ then
	
	 $(\bT(g)(e), \bT(g \ell_u^{c_u} (g))(e), \bT(g \ell_v^{c_v} (g))(e), \bT(g \ell_u^{c_u} (g)^{-1})(e), \bT(g \ell_v^{c_v} (g)^{-1})(e))$\\ is a valid pattern of $W$.
\end{enumerate}
$\mathcal Z$ is clearly an SFT. Also define 
\begin{align*}
 \beta : \mathcal Z &\to \{0,1\}^{\N}\\
  (\rho,\bB,\vartextvisiblespace,\vartextvisiblespace,\bL,\vartextvisiblespace,\vartextvisiblespace) &\mapsto (\bL(\gamma_{1_\Gamma}^v (n+1,\rho))(e_{1_\Gamma},\bB(1_\Gamma))(1_\Gamma))_{n \in \N} \mbox{ where } v \in \bB(1_\Gamma). 
\end{align*}

The following claim proves that the choice of $v \in \bB(1_\Gamma)$ is inconsequential.\
\begin{claim}
\label{claim:Indifferent}
For every $u,v \in \bB(1_\Gamma)$, and for every $(\rho,\bB,\vartextvisiblespace,\vartextvisiblespace,\bL,\vartextvisiblespace,\vartextvisiblespace) \in \mathcal{Z}$,\\ $(\bL(\gamma_{1_\Gamma}^v (n+1,\rho))(e_{1_\Gamma},\bB(1_\Gamma))(1_\Gamma))_{n \in \N} = (\bL(\gamma_{1_\Gamma}^u (n+1,\rho))(e_{1_\Gamma},\bB(1_\Gamma))(1_\Gamma))_{n \in \N}. $
\end{claim}
\begin{proof}
  By the second condition of the bush subshift, $\bB(1_\Gamma)$ is connected, and so it suffices to show it for two adjacent vertices. Let us hence assume that $v \in \lk(u)$. 
  
  If $v=u$, the result is trivial. 
  
  If not, by conditions 4 and 5 of the bush subshift, $ \forall n,m \in \N$,$\{u,v\} \in \bD(\gamma_{1_\Gamma}^u (n,\rho) \gamma_{\gamma_{1_\Gamma}^u (n,\rho)}^v (m,\rho))$. But then by condition 6, it is straightforward to show by induction that $\forall n \in \N, \forall k \leq n$,\\
  $\bL(\gamma_{1_\Gamma}^u (n-k,\rho) \gamma_{\gamma_{1_\Gamma}^u (n-k,\rho)}^v (k,\rho))(\{u,v\},\bB(g)) = \bL(\gamma_{1_\Gamma}^u (n,\rho))(\{u,v\},\bB(g))$, and the result follows from the case of $k=n$.
\end{proof}

In the following, we will prove that the action $F_S \curvearrowright X'$ is a topological factor of $\mathcal Z$ through $\beta$.

\begin{claim}
\label{claim:Incl}
$\beta(\mathcal Z) \subset X.$
\end{claim}

\begin{proof}
By condition 2 of the computation subshift, there exists 
  
  $e_{1_\Gamma} = ((u,\overline{c_u (g)}),(v,\overline{c_v (g)})) \in \bP(1_\Gamma)$ such that $u,v \in \bB(1_\Gamma)$. But then by condition 3, $\forall n,m \in \N, e_{1_\Gamma} \in \bP(\gamma_{1_\Gamma}^u (n,\rho) \gamma_{\gamma_{1_\Gamma}^u (n,\rho)}^v (m,\rho))$. By condition 4, $\bT(1_\Gamma)(e_{1_\Gamma}) = \bseed$ and by condition 6, $\bT$ defines a valid tiling at every $n,m \in \N$. But any valid $W$-tiling of the quarter plane with $\bseed$ at the origin must have an element of $Y$ as a first row, and so $(\bT(\gamma_{1_\Gamma}^u (n+1, \rho))(e_{1_\Gamma}))_{n \in \N} \in Y$. 
  
  Finally, note that $\forall n \in \N, ((u,\overline{c_u(g)}),\bB(1_\Gamma)) \in \bI(\gamma_{1_\Gamma}^u (n+1, \rho))$ and \\
  $\bT(\gamma_{1_\Gamma}^u (n+1, \rho))(e_{1_\Gamma}) \neq \bseed$. 
  
  Hence, by the 5th condition, \\
  $((\bL(\gamma_{1_\Gamma}^u (n+1,\rho))(e_{1_\Gamma}, \bB(1_\Gamma))(1_\Gamma))_{n \in \N} = ((\bT(\gamma_{1_\Gamma}^u (n+1,\rho))(e_{1_\Gamma})(1_\Gamma))_{n \in \N} \in X.$
  
  On the other hand, since $\mathcal Z$ is a $\Gamma$-subshift, any shift by a relation of $\Gamma$ of course acts trivially, so $\beta(\mathcal{Z}) \subset X'$.
\end{proof}

\begin{claim}
\label{cl:Equiv}
$\beta$ is $F_S$-equivariant.
\end{claim}

\begin{proof}
  Let $v \in V, s \in S_v \setminus \{F_S\}$. 
  By condition 2 of the bush subshift, there exists $u \in \bB(1_\Gamma) \cap \bB(s) \cap \lk(v)$. Then, conditions 7 and 8 ensure that $\forall n \in \N: u \in \bI(\gamma_{1_\Gamma}^v (n, \rho))$. By condition 9 since $v \in \lk(u)$,\\
  $ (\bL(\gamma_{1_\Gamma}^u (n, \rho) s)(e,\bB(s))(1_\Gamma))_{n \in \N} = (\bL(\gamma_{1_\Gamma}^u (n, \rho))(f,\bB(1_\Gamma))(s^{-1})))_{n \in \N}$ every time $v \in e \subset \bB(1_\Gamma)$ and $v \in f \subset \bB(s)$. Hence, 
  \begin{align*}
   	s \beta((\rho,\bB,\bD,\bI,\bL,\bP,\bT)) &= s(\bL(\gamma_{1_{F_S}}^v (n+1, \rho))(e_{1_\Gamma},\bB(1_\Gamma))(1_\Gamma))_{n \in \N}\\
  	&= (\bL(\gamma_{1_\Gamma}^v (n+1, \rho))(e_{1_\Gamma},\bB(1_\Gamma))(s))_{n \in \N} \mbox{ by Claim }~\ref{claim:Incl}.\\
  	&= (\bL(\gamma_{1_\Gamma}^u (n+1, \rho))(e_{1_\Gamma},\bB(1_\Gamma))(s))_{n \in \N} \mbox{ by Claim }~\ref{claim:Indifferent}.\\
  	&= (\bL(\gamma_{1_\Gamma}^u (n+1, \rho) s^{-1})(e_{s^{-1}},\bB(s^{-1}))(1_\Gamma))_{n \in \N}\\
	&=(\bL(s^{-1}\gamma_{1_\Gamma}^u (n+1, \rho))(e_{s^{-1}},s\bB(1_\Gamma))(1_\Gamma))_{n \in \N} \mbox{ as } s \mbox{ commutes with } \Gamma_u.\\
	&= (s\bL)(\gamma_{1_\Gamma}^u(n+1,s\rho))(e_{s^{-1}},s\bB(1_\Gamma))(1_\Gamma))_{n \in \N}\\
  	&= \beta((s\rho, s \bB,s \bD,s \bI, s \bL, s \bP, s \bT)). 
\end{align*}
  This concludes the proof as $S$ generates $F_S$.
\end{proof}

\begin{claim}
\label{claim:Incl}
$\beta(\mathcal Z) \subset X'.$
\end{claim}

\begin{proof}
Since $\mathcal{Z}$ is a $\Gamma$-subshift, relators of $\Gamma$ of course act trivially on $\mathcal{Z}$, so by $F_S$-invariance of $\beta$, the image is in $X'$.
\end{proof}

\begin{claim}
\label{claim:Surj}
If $G$ is atomic, then $\beta : \mathcal Z \to X'$ is surjective.
\end{claim}

\begin{proof}
Let $x \in X'$, and for all $g \in \Gamma$, let $y^{(g)} \in Y$ be such that $(y_n^{(g)}(1_{F_S}))_{n \in \N} = \hat g^{-1} x$ where $\hat g \in F_S$ is any element that evaluates to $g$ in $\Gamma$. Note that since $x \in X'$, $\hat g^{-1}x$ does not depend on the choice of $\hat g$. Define $\rho, \bB, \bD,\bI$ and $\bL$ as in Claim~\ref{claim:NonEmpty}. 
  
  Now, for every $g \in \Gamma$, choose one edge $\{u_g,v_g\} \subset \bB(g)$, and set $e_g = ((a_g,\overline{c_{a_g} (g)}),(v,\overline{c_{b_g} (g)}))$. 
  
  Define $\bP$ by $\forall e = ((u,c_u),(v,c_v)) \in \cE, e \in \bP(g) \mbox{ if and only if } \exists n,m \in \N, \exists h \in \Gamma \mbox{ such that } g = h \gamma_h^u (n,\rho) \gamma_{h \gamma_h^u (n,\rho)}^v (m,\rho)$ and $e_h = e$.
  
  Also define $\bT(g)(e_g) = \bseed$ and $\forall n\in \N,\bT(g \gamma_g^{u_g}(n+1,\rho))(e_g) = y^{(g)}_{n}$. Extend $\bT$ so that $\bT(g \gamma_g^{u_g} (n,\rho) \gamma_{g \gamma_g^{u_g} (n,\rho)}^{v_g} (m,\rho))(e_g)$ is defined by using the compatibility conditions of $W$. This definition is justified by the same remark as the definition of $\bL$ in Claim~\ref{claim:NonEmpty}. Now extend $\bT$ to every layer on every point to satisfy condition 5. 
  
  Note that this condition only applies to elements of the form $g_0 \gamma_{g_0}^u (n,\rho)$ for some $n>1$. Finally, extend arbitrarily $\bT$ to every layer on every point. 
  
  We claim that $(\rho,\bB,\bD,\bI,\bL,\bP,\bT) \in \mathcal Z$.
  \begin{enumerate}
  \item This follows from Claim~\ref{claim:NonEmpty}.
  \item This follows from the definition of $\bP$ with $h = g, n=m=0.$
  \item This follows from the definition of $\bP$ with $g = g_0 \gamma_{g_0}^u (n,\rho) \gamma_{g_0 \gamma_{g_0}^u (n,\rho)}^v (m,\rho)$, as then $g \ell_u^{c_u} (g) = g_0 \gamma_{g_0}^u (n+1,\rho) \gamma_{g_0 \gamma_{g_0}^u (n,\rho)}^v (m,\rho)$, and\\
  $g \ell_b^{c_v} (g) = g_0 \gamma_{g_0}^u (n,\rho) \gamma_{g_0 \gamma_{g_0}^u (n,\rho)}^v (m+1,\rho)$.
  \item This is part of the definition of $\bT$.
  \item This follows from the definition of $\bP,\bL$ and $\bT$ when $g = g_0 \gamma_{g_0}^u (n,\rho)$, as then $\bL(g) (\{u,v\})(C) = \bL (g)(\{u,v\}) (\bB(g_0)) = y^{(g)}_{n-1} = \bT(g)(e_{g_0}) =  \bT(g)(e)$. But $\bI$ is empty outside of these, and so the condition is verified everywhere.
  \item This is part of the definition on every $g_0 \gamma_{g_0}^{u_{g_0}} (n,\rho) \gamma_{g_0 \gamma_{g_0}^{u_g} (n,\rho)}^{v_{g_0} }(m,\rho)$. But $\bP$ is empty outside of these, and so the condition is verified everywhere.
  \end{enumerate}
But finally, from the definition of $\bL$ it is clear that $\beta((\rho,\bB,\bD,\bI,\bL,\bP,\bT)) = x$, and so $\beta$ is surjective.
\end{proof}

Theorem~\ref{th:GraphProducts} is now clear.

\begin{proof}[Proof of Theorem~\ref{th:GraphProducts}]
Let $G$ be a graph that is not a clique or contains at least two non-amenable vertices.

  If $G$ has no disconnecting amenable clique, then it is atomic and so for any effectively closed action $\Gamma(G) \curvearrowright X \subset \{0,1\}^\omega$, Claims~\ref{claim:Incl},~\ref{cl:Equiv} and~\ref{claim:Surj} show that $\Gamma \curvearrowright X$ is a topological factor of the subshift of finite type $\mathcal Z$ through $\beta$.
  
  Conversely, if $G$ has a disconnecting amenable clique, then by Corollary~\ref{cor:AmenableClique}, it is not self-simulable.
\end{proof}

We then obtain Theorem~\ref{th:RAAGs} as a corollary of Theorem~\ref{th:GraphProducts}.

\begin{proof}[Proof of Theorem~\ref{th:RAAGs}.]
	This is an immediate consequence of Theorem~\ref{th:GraphProducts} in the case where every vertex is infinite and amenable.
\end{proof}


\section{Proof of Proposition~\ref{prop:CliqueGeometryThing}}
\label{sec:CliqueGeometryThing}

\begin{repproposition}{prop:CliqueGeometryThing}
Let $G$ be a clique of finitely-generated groups, with a single non-amenable vertex $\Delta$. Then $\Gamma(G)$ splits non-trivially over an amenable subgroup if and only if $\Delta$ does.  
\end{repproposition}

\begin{proof}
In this case, there is an amenable group $\Gamma$ such that $\Gamma(G) = \Gamma \times \Delta$, with all groups finitely generated. 

If $\Delta$ splits nontrivially over an amenable subgroup, then there is an action of $\Delta$ on a tree with no globally fixed vertices, no edge inversions, and all edge stabilizers amenable. Then the action of $\Gamma \times \Delta$ where $\Gamma$ acts trivially has the same properties.

Conversely, suppose $\Gamma \times \Delta$ splits nontrivially over an amenable group, so it admits an action on a tree $T$ which does not have any globally fixed vertex, has no edge inversions, and edge stabilizers are amenable.

If the action of $\Delta$ has no global fixed point, then the subaction of $\Delta$ on $T$ acts on the same tree has all the same properties as the original action, so $\Delta$ splits nontrivially over an amenable group.

Suppose then that the action of $\Delta$ has a global fixed point. Let $T_1$ be the set of its fixed points. Then $T_1$ is itself a subtree of $T$. If $T_1$ is not a single vertex, then $\Delta$ fixes an edge of this tree, which is impossible since edge stabilizers were assumed to be amenable. So $T_1$ has a single vertex.

Since our group $\Gamma(G)$ is a direct product of $\Gamma$ and $\Delta$, $T_1$ has to be fixed as a set by the amenable group $\Gamma$ as well, since if $g \in \Gamma$, $\Delta = \Delta^g$ has fixed points $g^{-1}T_1$. So the unique vertex of $T_1$ is also fixed by $\Delta$, meaning the original action has a globally fixed vertex, a contradiction.
\end{proof}

\begin{remark}
Note that in this proof, the splitting is of the same kind, so $\Gamma \times \Delta$ splits as an HNN extension over an amenable group (resp.\ as an amalgamated free product over an amenable group) if and only if the direct product does.
\end{remark}

\section{Questions}

Our self-simulability construction requires the node groups $G_u$ to be infinite. In fact, bushes may not be generalized easily to the case of finite groups. Indeed, the graph product of a square of $\Z_3$s is self-simulable because it is a direct product of two non-amenable groups. However, the complement of an edge is another edge, which is finite and hence cannot be used as a bush.

\begin{question}
If we allow finite node groups, when is a graph product self-simulable?
\end{question}

In particular, we can ask the following if every node group is $\Z_2$.

\begin{question}
Which right-angled Coxeter groups are self-simulable?
\end{question}

The characterization is not the same as with RAAGs. Indeed, as we showed in Example~\ref{ex:Cycle}, a cycle of length at least $4$ always defines a self-simulable RAAG. However, for large enough cycles, any two disconnected nodes give a copy of the amenable group $D_{\infty} = \Z_2 * \Z_2$ that disconnects the group. On the other hand, we do not know if the triangular prism graph of Figure~\ref{fig:Toblerone} generates a self-simulable right-angled Coxeter group. Indeed, it has no disconnecting amenable subgraph, but the complement of the red clique generates a finite group, so that the bush method cannot work as is.

If a right-angled Coxeter group is quasi-isometric to a right-angled Artin group, then because both are finitely-presented, \cite[Theorem~1.5]{BaSaSa26} implies that their self-simulability statuses are the same. It seems to be unknown which right-angled Coxeter groups have this property, but partial results are given in \cite{CaDaEdKa25}.

\begin{figure}[ht!]
	\centering
	\begin{tikzpicture}
		\node[shape=circle,fill=black] (A) at (0,0) {};
		\node[shape=circle,fill=black] (B) at (0,1) {};
		\node[shape=circle,fill=black] (C) at (0.85,0.5) {};
		\node[shape=circle,fill=red] (D) at (3,0) {};
		\node[shape=circle,fill=red] (E) at (3,1) {};
		\node[shape=circle,fill=red] (F) at (2.15,0.5) {};
		\path[-] (A) edge node[left] {} (D);
		\path[-] (B) edge node[left] {} (E);
		\path[-] (C) edge node[left] {} (F);
		\path[-] (A) edge node[left] {} (B);
		\path[-] (B) edge node[left] {} (C);
		\path[-] (C) edge node[left] {} (A);
		\path[-] (D) edge node[left] {} (E);
		\path[-] (E) edge node[left] {} (F);
		\path[-] (F) edge node[left] {} (D);
	\end{tikzpicture}
	\caption{Prism graph.}
	\label{fig:Toblerone}
\end{figure}

\bibliographystyle{plain}
\bibliography{bib}{}

\end{document}